\documentclass[11pt]{amsart}

\usepackage{amssymb, amscd}
\usepackage{eucal}
\usepackage{color}

\usepackage{soul}

\DeclareSymbolFont{SY}{U}{psy}{m}{n}
\DeclareMathSymbol{\emptyset}{\mathord}{SY}{'306}

\theoremstyle{plain}

\newtheorem{thm}{Theorem}[section]

\newtheorem{lem}[thm]{Lemma}
\newtheorem{prop}[thm]{Proposition}

\theoremstyle{definition}

\newtheorem{rem}[thm]{Remark}

\DeclareMathOperator{\Ad}{Ad}
\DeclareMathOperator{\ad}{ad}

\DeclareMathOperator{\Hol}{Hol}
\DeclareMathOperator{\Hom}{Hom}
\DeclareMathOperator{\Aut}{Aut}

\newcounter{defcounter}
\newenvironment{myequation}{%
\addtocounter{equation}{-1}
\refstepcounter{defcounter}

\begin{equation}}
{\end{equation}}

\title[Homogeneous bundles and operators in the Cowen-Douglas class]
{Homogeneous Hermitian holomorphic vector bundles and the Cowen-Douglas class over bounded symmetric domains 
}

\author[A. Kor\'{a}nyi]{Adam Kor\'{a}nyi} 
\author[G. Misra]{Gadadhar Misra}

\address[A. Kor\'{a}nyi]{Lehman College\\
Bronx, NY 10468}
\address[G. Misra]{Department of Mathematics\\Indian Institutte of Science\\Bangalore 560012}
\email[A. Kor\'{a}nyi]{adam.koranyi@lehman.cuny.edu}
\email[G. Misra]{gm@math.iisc.ernet.in}
\keywords{homogeneous Hermitian holomorphic vector bundles, holomorphic induction, Cowen-Douglas class} 
\subjclass[2010]{Primary 47A13, 20C25; Secondary 32M15, 53C07}

\thanks{Both the authors were supported, in part, by a DST - NSF S\&T Cooperation Program and the J C Bose National Fellowship of the Department 
of Science and Technology. The second author also gratefully acknowledges the support from the University Grants Commission Centre for Advanced Studies.}

\begin{document}

\begin{abstract}
It is known that all the vector bundles of the title can be obtained by holomorphic induction from representations of a certain parabolic group on finite dimensional inner product spaces. The representations, and the induced bundles, have composition series with irreducible factors.  We give a condition under which the bundle and the direct sum of its irreducible constituents are intertwined by an equivariant constant coefficient differential operator. We show that in the case of the unit ball in $\mathbb C^2$ this condition is always satisfied.
As an application we show that all homogeneous pairs of Cowen-Douglas operators are similar to direct sums of certain basic pairs.
\end{abstract}
\maketitle
 
\section{Holomorphic vector bundles}
Let $\mathfrak g$ be a simple non-compact Lie algebra with Cartan decomposition $\mathfrak g = \mathfrak k + \mathfrak p$ such that $\mathfrak k$ is not semi-simple. Then $\mathfrak k$ is the direct sum of its center and  of its semisimple part, $\mathfrak k = \mathfrak z + \mathfrak k_{\rm ss},$ and there is an element $\hat{z}$ which generates $\mathfrak z$ and $\ad(\hat{z})$ is a complex structure on $\mathfrak p.$

The complexification $\mathfrak g^\mathbb C$ is then the direct sum $\mathfrak p^+ + \mathfrak K^\mathbb C + \mathfrak p^-$ of the $i,0,-i$ eigenspaces of $\ad(\hat{z})$ 
We let $G^\mathbb C$ denote the simply connected Lie group with Lie algebra $\mathfrak g^\mathbb C$ and we let $G, K^\mathbb C, K, P^{\pm}, \ldots $ be the analytic subgroups corresponding to $\mathfrak g^\mathbb C, \mathfrak k^\mathbb C, \mathfrak k, \mathfrak p^{\pm},\ldots.$
We denote by $\tilde{G}$ the universal covering group of the group $G$ and by $\tilde{K},\tilde{K}_{\rm ss},\ldots $ its analytic subgroups corresponding to $\mathfrak k,\mathfrak k_{\rm ss},\ldots .$

$K^\mathbb CP^-$ is a parabolic subgroup of $G^\mathbb C.$ $P^+K^\mathbb CP^-$ is open dense in $G^\mathbb C.$ The corresponding decomposition $g^+g^0g^-$ of any $g$ in $P^+K^\mathbb CP^-$ is unique. The natural map $G/K \to G^\mathbb C/K^\mathbb C P^-$ is a holomorphic imbedding, its image is in the orbit of $P^+.$   Applying now $\exp_{\mathfrak p^+}^{-1}$ we get the Harish-Chandra realization of $G/K$ as a bounded symmetric domain $\mathcal D\subset\mathfrak p^+.$ The action of $g\in G$ on $z\in \mathcal D,$ written $g\cdot z,$  is then defined by $\exp(g\cdot z) = (g \exp z)^+.$ 
We will use the notation $k(g,z) = (g \exp z)^0$ and $\exp Y(g,z) = (g \exp z)^-,$ so we have  
$$
g\exp z = (\exp(g\cdot z)) k(g,z) \exp (Y(g,z)).
$$

%
The $\tilde{G}$ - homogeneous Hermitian holomorphic vector bundles ({\tt hHhvb}) over $\mathcal D$  are 
obtained by holomorphic induction from representations $(\rho,V)$ of $\mathfrak k^\mathbb C + \mathfrak p^-$ on finite dimensional inner product spaces $V$ such that $\rho(\mathfrak k)$ is skew Hermitian.   
We write $\rho^0,\,\rho^{-}$ for the restrictions of $\rho$ to $\mathfrak k^\mathbb C$ and $\mathfrak p^-,$ respectively. 
The representation space $V$ is the orthogonal direct sum of its subspaces $V_\lambda$ ($\lambda \in \mathbb R$) on which $\rho^0(\hat{z})=i \lambda.$ It is easy to see that $\rho^-(Y) V_\lambda \subset V_{\lambda-1}$ for $Y\in \mathfrak p^-.$ We also have 
\begin{equation} \label{1}
\rho^-([Z,Y]) = [\rho^0(Z), \rho^-(Y)], \,\,Z\in \mathfrak k^\mathbb C,\, Y\in \mathfrak p^-.
\end{equation}

We note that  if  representations  $\rho^0$ and $\rho^-$ of $\mathfrak k^\mathbb C$ and $\mathfrak p^-,$ respectively, are given, then they will together give a representation of $\mathfrak k^\mathbb C + \mathfrak p^-$ if and only if  equation \eqref{1} holds. We call $(\rho, V)$ and the induced bundle, indecomposable if it is not the orthogonal sum of sub-representations, respectively, sub-bundles. We  restrict ourselves to describing these.
\begin{prop} \label{prop 1.1}
Every indecomposable holomorphic homogeneous Hermitian vector bundle $E$ can be written as a tensor product $L_{\lambda_0} \otimes E^\prime$, where $L_{\lambda_0}$ is the line bundle induced by a character $\chi_{\lambda_0}$ and $E^\prime$ is the lift to $\tilde{G}$ of a $G$ - homogeneous  holomorphic  Hermitian vector bundle which is the restriction to $G$ and $\mathcal D$ of a $G^\mathbb C$ - homogeneous vector bundle induced in the holomorphic category by a representation of $K^\mathbb CP^-$.
\end{prop}
The proof involves some structural properties of $G^\mathbb C,$ which we omit in this short Announcement. 


As shown in \cite{HW}, $P^+ \times \tilde{K}^\mathbb C \times P^-$ can be given 
a structure of complex analytic local group such that (writing $\pi: \tilde{K}^\mathbb C \to K^\mathbb C$)
${\rm id} \times \pi \times{\rm id}$ is the universal local group covering of $P^+K^\mathbb C P^-.$
We write $\tilde{G}_{\rm loc}$ for this local group and abbreviate ${\rm id} \times \pi \times{\rm id}$ to $\pi.$  
By \cite{HW}, $\tilde{G},\, \tilde{K}^\mathbb C P^-,$ $P^+ \tilde{K}^\mathbb C$ are closed subgroups of 
$\tilde{G}^\mathbb C_{\rm loc}$ and $\tilde{G} \exp \mathcal D \subset \tilde{G}^\mathbb C_{\rm loc}.$
Defining $g\cdot z = \pi(g) \cdot z$ and $Y(g,z) = Y(\pi(g),z)$ we have the decomposition
$$
g \exp z = (\exp g\cdot z) \tilde{k}(g,z)\exp Y(g,z), \,\, (g\in \tilde{G}, z\in\mathcal D)
$$
in $\tilde{G}_{\rm loc}.$ We write $\tilde{b}(g,z) = \tilde{k}(g,z) \exp Y(g,z);$ then $\tilde{b}(g,z)$ satisfies the 
multiplier identity and $\tilde{b}(kp^-,0) = k p^-$ for $kp^- \in \tilde{K}^\mathbb C P^-.$  

Hence given a representation  $(\rho,V)$ of $\mathfrak k^\mathbb C + \mathfrak p^-$ as above, the 
holomorphically induced bundle has a canonical trivialization such that the sections are the elements of $\Hol(\mathcal D,V),$ 
and $\tilde{G}$ acts via the multiplier 
$$
\rho(\tilde{b}(g,z)) = \rho^0(\tilde{k}(g,z))\rho^-(\exp Y(g,z)).
$$ 

If $f\in\Hol(\mathcal D,V),$ then we write $Df$ for the derivative: $Df(z) X = (D_Xf)(z)$ for $X\in \mathfrak p^+.$ 
Thus $Df(z)$ is a $\mathbb C$ - linear map from $\mathfrak p^+$ to $V.$ 
\begin{lem} \label{lem 1.2}
For any holomorphic representation $\tau$ of $\tilde{K}^\mathbb C$ and any $g\in \tilde{G}$, $z\in\mathcal D$, $X\in \mathfrak p^+$,  
$$D_X\tau\big (\tilde{k}(g,z)^{-1}\big ) = - \tau\big ([ Y(g,z), X]\big )\tau\big (\tilde{k}(g,z)^{-1}\big ).$$
Furthermore,
$$
D_X\,Y(g,z)=\frac{1}{2} \big [Y(g,z),[Y(g,z),X]\big ].
$$
\end{lem}
This is proved by refining the arguments of \cite[p. 65]{S} 
\section{The main results about vector bundles}
If in the set up of Section 1 each subspace  $V_\lambda$ is irreducible under $\mathfrak k^\mathbb C$ we call the corresponding representations and the vector bundles {\em filiform}. We consider this case first.

We have seen that  every indecomposable filiform representation is a  direct sum of subspaces $V_{\lambda-j},$ which we denote by $V_j,$ carrying an irreducible representation $\rho_j^0$ of $\mathfrak k^\mathbb C\:(0 \leq j \leq m),$ furthermore, we have non-zero $\mathfrak k^\mathbb C$-equivariant maps $\rho_j^-:\mathfrak p^-\to \Hom(V_{j-1}, V_j).$ The space of such maps is $1$-dimensional: This is an equivalent restatement of the known fact that $\mathfrak p^-\otimes V_{j-1}$ as a representation of $\mathfrak k^\mathbb C$ is multiplicity free \cite[Corollary 4.4]{Jak}. We denote the orthogonal projection from $\mathfrak p^- \otimes V_{j-1}$ to $V_j$ by $P_j.$ 
We define for $Y\in \mathfrak p^-,\,v\in V_{j-1},$
\begin{equation} \label{2}
\tilde{\rho}_j(Y)v=P_j(Y\otimes v).
\end{equation}
Then $\tilde{\rho}_j$ has the $\mathfrak k^\mathbb C$-equivariant property, and it follows that $\rho_j^- = y_j \tilde{\rho}_j$ with some $y_j\not = 0.$ We write $y=(y_1, \ldots, y_m)$ and denote by $E^y$ the induced vector bundle. We observe here that the vector bundle $E^y$ is uniquely determined by $\rho_0^0, P_1,\ldots , P_m$ and $y,$ but these data cannot be arbitrarily chosen: The $\tilde{\rho}_j\:(1 \leq j \leq m)$ together must give a representation of the abelian Lie algebra $\mathfrak p^-.$ In terms of $P_j,$ this condition amounts to 
\begin{equation}\label{3}
P_{j+1}\big (Y^\prime \otimes P_j(Y\otimes v) \big )= P_{j+1}\big (Y \otimes P_j(Y^\prime\otimes v) \big )
\end{equation}
for all $Y,Y^\prime \in \mathfrak p^-$ and $v\in V_{j-1}.$

We denote by $\iota$ the identification of $(\mathfrak p^+)^*$ with $\mathfrak p^-$ under the Killing form, and for any vector space $W,$ extend it to a map from $\Hom(\mathfrak p^+, W)$ to $\mathfrak p^-\otimes W,$ that is, for $Y\in \mathfrak p^-, w\in W,$ 
$$
\iota(B(\cdot,Y)w) = Y\otimes w.
$$
\begin{lem} \label{lem 2.1}
Given $\rho^0_{j-1}, \rho^0_{j},$ as above, there exists a 
constant $c_j,$ independent of $\lambda,$ such that for all $Y\in \mathfrak p^-$, we have
$$
P_j \iota \rho^0_{j-1}([Y,\cdot]) = (c_j- \tfrac{\lambda-j+1}{2n}) \tilde{\rho}_j(Y).
$$
\end{lem}
This follows from $\mathfrak k^\mathbb C$-equivariance and a character computation. The following two lemmas can be proved by computations based on Lemmas \ref{lem 1.2} and \ref{lem 2.1}.
\begin{lem}
For all $1 \leq j \leq m-1,$ and holomorphic $F:\mathcal D\to V_j$,
\begin{eqnarray*}
\lefteqn{P_{j+1} \iota D^{(z)} \big (\rho_j^0 (\tilde{k}(g,z)^{-1})F(g z) \,\big)}\\
&=& -(c_{j+1}-\tfrac{\lambda-j}{2n})\tilde{\rho}_{j+1} \big (Y(g,z)\big )\big (\rho_j^0 (\tilde{k}(g,z)^{-1})F(g z) \,\big) + \rho_{j+1}^0 (\tilde{k}(g,z)^{-1}) \big ( (P_{j+1} \iota D F)(gz)\big ),\end{eqnarray*}
where $D^{(z)}$ denotes the differentiation with respect to $z.$ 
\end{lem}

\begin{lem}
For all $1 \leq j \leq m-1,$ with the constants $c_j$  of Lemma \ref{lem 2.1}, 
$$
P_{j+1}\iota D^{(z)} \tilde{\rho}_j(Y(g,z)) = \frac{1}{2} (c_j - c_{j+1}-\tfrac{1}{2n}) \tilde{\rho}_{j+1} (Y(g,z))\tilde{\rho}_j(Y(g,z)).
$$ 
\end{lem}
Now let $E^y$ be an indecomposable filiform hHhvb as described above. Writing $0=(0,\ldots,0),$ $E^0$ makes sense, it is the direct sum of irreducible  vector bundles in the composition series of $E^y.$  

If $f\in \Hol(\mathcal D, V),$ we write $f_j$ for the component of $f$ in $V_j,$ that is, the projection of $f$ onto $V_j.$

\begin{thm} \label{2.4}
Assume that in $E^y,$ the constants $c_j$ of Lemma \ref{lem 2.1} are of the form 
\begin{myequation}\label{myeqn:1}
c_j = u + (j-1) v,\:1\leq j \leq m,
\end{myequation} 
with some constants $u, v$ and $\lambda$ is regular in the sense that 
$$
c_{jk}=\tfrac{1}{(j-k)!} \prod_{i=1}^{j-k} \big \{ u - \tfrac{\lambda}{2n}+ \tfrac{2k+i-1}{2}\big (v+\tfrac{1}{2n} \big )\big \}^{-1}
$$
is meaningful for $0\leq k \leq j \leq m.$
Then the operator $\Gamma: \Hol(\mathcal D, V) \to \Hol(\mathcal D, V)$ given by 
$$
(\Gamma f_j)_\ell  = \begin{cases}  c_{\ell j}\, y_\ell \cdots y_{j+1} (P_\ell \iota D)  \cdots (P_{j+1} \iota D) f_j & \mbox{\!\!\!\rm if~} \ell > j,\\
f_j \:\mbox{~\rm if~} \ell = j, \\
0 \:\:\:\mbox{~\rm if~}\ell< j&
\end{cases}$$
intertwines the actions of $\tilde{G}$ on the trivialized sections of $E^0$ and $E^y$.
\end{thm}
The proof is by induction based on the preceding lemmas.

We note that condition \eqref{myeqn:1} is vacuous when $m=1$ or $2.$ Otherwise, it can be shown that \eqref{myeqn:1} is also necessary for the theorem to hold.  

Next we pass from the filiform case to the general case. Now $(\rho_0,V_0)$ is a direct sum of representations $({\rho_j^0}^\alpha, V_j^\alpha)$ with inequivalent irreducible representations $\alpha$ of $\mathfrak k_{\rm ss}^\mathbb C,$ and ${\rho_j^0}^\alpha=\chi_{\lambda-j}(I_{m_{j\alpha}} \otimes \alpha).$ For pairs of $(\alpha, \beta)$ that are admissible in the sense that$\beta \subset \Ad_{\mathfrak p^-} \otimes \alpha$ we write $P_{\alpha \beta}$ for the corresponding projection and define maps $\tilde{\rho}_{\alpha \beta}$ for $Y\in \mathfrak p^-.$ Then
$$
\rho_j^-(Y) = \oplus_{\alpha,\beta} y_j^{\alpha \beta} \otimes \tilde{\rho}_{\alpha\beta}(Y)
$$
with $y_j^{\alpha \beta} \in \Hom(\mathbb C^{m_\alpha},\mathbb C^{m_\beta})$ such that $y_{j+1}^{\beta\gamma}y_{j}^{\alpha\beta}=0$ unless $(\alpha\beta)$ and $(\beta\gamma)$ are admissible and the analogue of \eqref{3} holds. We let $\mathbb E^y$ denote the bundle holomorphically induced by $\rho,$ and let $\mathbb E^0$ be the (direct sum) bundle gotten by changing all the $y^{\alpha\beta}$ to $0.$ The general version of $\Gamma$ is now going to be  (for $j < \ell$)
$$
(\Gamma f_j)_\ell = {\oplus}_{{{\alpha_j}, \ldots, {\alpha_\ell}}} c_{\ell j}^{{\alpha_j}, \ldots, {\alpha_\ell}} \big (y_\ell^{{\alpha_{\ell-1}} {\alpha_\ell}} \cdots y_{j+1}^{{\alpha_j}{\alpha_{j+1}}}\big ) \otimes  \big (P_{{\alpha_{\ell-1}}{\alpha_\ell}} \iota D \big )                 
\cdots \big (P_{{\alpha_{j}}{j+1}} \iota D \big ) f_j^\alpha.                 
$$
For $j \geq \ell,$ it is unchanged. 

\begin{thm}\label{thm 2.5}
Suppose that $(\#)$ holds for all indecomposable filiform subbundles of $\mathbb E^y.$ Then there exist constants $c_{\ell j}^{{\alpha_j}, \ldots, {\alpha_\ell}}$ such that $\Gamma$ intertwines the actions of $\tilde{G}$ on the trivialized sections of $\mathbb E^0$ and $\mathbb E^y.$
\end{thm}

We note that if $\mathcal D$ is the disc in one variable then $(\#)$ always holds with $u=v=0.$ So, Theorem \ref{thm 2.5} contains Theorem 3.1 of \cite{KM}. We proceed to discuss cases where \eqref{myeqn:1} is satisfied. 


The Cartan product of two irreducible representations of $\mathfrak k^\mathbb C$ is the irreducible component of the tensor product whose highest weight is the sum of the highest weights of the original representations. 
\begin{lem}\label{lem 2.5}
Let $\rho_0^0$ be any irreducible representation of $\mathfrak k^\mathbb C$ and define $\rho_j^0\:(1\leq j \leq m)$ inductively as the Cartan product of $\Ad_{\mathfrak p^-}$ and $\rho^0_{j-1}.$ Then with $\tilde{\rho}_j$ as in \eqref{2} and $\rho_j^- = y_j\tilde{\rho}_j\: (y_j \not = 0)$ we obtain a filiform representation $\rho$ of $\mathfrak k^\mathbb C + \mathfrak p^-.$ In this case \eqref{myeqn:1} is automatically satisfied. 
\end{lem}
Both the statements in this Lemma are proved by using weight theory. 

Finally, we specialize to the case where $\mathcal D$ is the unit ball in $\mathbb C^2.$ Then $G=SU(2,1),$ $\mathfrak k_{\rm ss}^\mathbb C=\mathfrak{sl}(2,\mathbb C).$ It is well known that the irreducible representations of $\mathfrak{sl}(2,\mathbb C)$ are just the symmetric tensor powers of the natural representation, which is equivalent to $\Ad^\prime_{\mathfrak p^-},$ the restriction of $\Ad_{\mathfrak p^-}$ to $\mathfrak k_{\rm ss}^\mathbb C.$
Consequently, a complete description of indecomposable representations of $\mathfrak k^\mathbb C + \mathfrak p^-$ is possible in this case.
In the following, an exponent in brackets, $[k],$ denotes the ${k}$-{th} symmetric tensor power. 

\begin{prop} \label{prop 2.7}
For the complex ball in $\mathbb C^2,$ there are only two types of indecomposable filiform representations of $\mathfrak k^\mathbb C+ \mathfrak p^-.$ For the first type, $\rho^0_{j} = \chi_{\lambda-j}  \otimes (\Ad^\prime_{\mathfrak p^-})^{[k+j]},\:0\leq j \leq m,$ with some $\lambda\in \mathbb R$ and $k,m\in \mathbb N.$ For the second type, $\rho_j^0=\chi_{\lambda-j}\otimes (\Ad^\prime_{\mathfrak p^-})^{[k-j]},\:0\leq j \leq m,$ with $\lambda \in \mathbb R,$ and $m \leq k.$  For both types the condition \eqref{myeqn:1} is satisfied.
\end{prop}

\begin{proof}[Sketch of the proof] By Clebsch-Gordan, for $j\geq 1,$
$$\Ad^\prime_{\mathfrak p^-} \otimes (\Ad^\prime_{\mathfrak p^-})^{[j]} = (\Ad^\prime_{\mathfrak p^-})^{[j+1]} \oplus (\Ad^\prime_{\mathfrak p^-})^{[j-1]},$$
So for each $P_j$ there are two possibilities, ``up'' or ``down''. One can show that unless all are up or all are down, the condition in equation \eqref{3} will fail. 

When all $P_j$ are ``up'' (the first type), we are in the situation of Lemma \ref{lem 2.5}, so \eqref{myeqn:1} holds. The second type is the contragredient of  a representation of the first type: it follows that \eqref{myeqn:1} holds again.
\end{proof}
\begin{thm}
In the case of the unit ball in $\mathbb C^2$ the conclusion of Theorem \ref{thm 2.5} holds for all indecomposable hHhvb-s with regular $\lambda.$   
\end{thm}
This is because Proposition \ref{prop 2.7} shows that the condition of Theorem \ref{thm 2.5} is satisfied.
\section{Hilbert spaces of sections}
With notations preserved, for general $\mathcal D,$ we consider first the case where $\rho$ is irreducible. Then automatically $\rho^0$ is also irreducible and
$\rho^-=0.$ We write $\rho^0 = \chi_\lambda\otimes \sigma,$ where $\sigma$ is an irreducible representation of $\mathfrak k_{\rm ss}.$ For every $\sigma,$ there is an (explicitly known) set of $\lambda$-s such that the sections of the corresponding holomorphically induced vector bundle have a $\tilde{G}$-invariant inner product. This is Harish-Chandra's holomorphic discrete series and its analytic continuation. In the canonical trivialization it gives Hilbert spaces $\mathcal H_{\rho^0}= \mathcal H_{\sigma,\lambda}$ which are known to have reproducing kernels $K_{\sigma,\lambda}(z,w)$. If we set 
$$
\tilde{\mathcal K}(z,w) = \tilde{k}(\exp - \bar{w},z),
$$  
(where $\bar{w}$ denotes conjugation with respect to $\mathfrak g$) we have, slightly extending \cite[Chap II, \S 5]{S}
$$K_{\sigma,\lambda}(z,w) = (\chi_\lambda \otimes\sigma)(\tilde{\mathcal K}(z,w)).$$
In particular, it is known that the inner product is \emph{regular} in the sense that all $K$ - types (i.e polynomials) have non-zero norm in $\mathcal H_{\sigma, \lambda}$ if and only if $\lambda < \lambda_\sigma$ for a certain known constant $\lambda_\sigma.$ 

In the following theorem we consider a bundle $\mathbb E^y$ as in Section 2. The corresponding $\mathbb E^0$ is then a direct sum of irreducible bundles as above. Its sections have a $\tilde{G}$-invariant inner product if and only if this is true for each summand. In this case, we have a Hilbert space $\mathcal H^0 = \oplus\mathcal H_{\rho^0_j}.$
\begin{thm} \label{3.1}
Suppose $(\#)$ holds for all filiform subbundles of $\mathbb E^y.$ Then the sections of $\mathbb E^y$ have a $\tilde{G}$-invariant regular inner-product if and only if the same is true for $\mathbb E^0.$ In this case, the map $\Gamma$ is a unitary isomorphism of $\mathcal H^0$ onto the Hilbert space $\mathcal H^y$ of sections of $\mathbb E^y.$ The space $\mathcal H^y$ (as well as $\mathcal H^0$) has a reproducing kernel.  
\end{thm}

For the proof one observes that $\Gamma$ has an inverse of the same form (only the constants $c_{jk}$ change). 
$\Gamma$ being a holomorphic differential operator, the image of $\mathcal H^0$ is also a Hilbert space of holomorphic functions with a reproducing kernel. One can verify that this is the sought after $\mathcal H^y.$  

\begin{thm} \label{3.2}
Suppose $\mathcal D$ is the unit ball in $\mathbb C^n.$ Let $\sigma_0, \sigma_1$ be the irreducible representations of $\mathfrak k^\mathbb C_{\rm ss}$ such that $\sigma_1 \subset \Ad_{\mathfrak p^-}\otimes \sigma_0$ and let $P$ be the corresponding projection.  Then if $\lambda < \lambda_{\sigma_0},$ we have $\lambda-1 < \lambda_{\sigma_1}$ and $P\iota D$ is a bounded linear transformation from $\mathcal H_{\sigma_0,\lambda}$ to $\mathcal H_{\sigma_1,\lambda-1}.$ 
\end{thm}
By the theory of reproducing kernels, for this it is enough to prove that ($D^{(z)}$ and $D^{(w)}$ denote the differentiation 
with respect to the variable $z$ and $w$ respectively)
$$
C K_{\sigma_1,\lambda-1}(z,w) - (P\iota D^{(z)})K_{\sigma_0,\lambda}(z,w) (P\iota D^{(w)})^*
$$
is positive definite for some $C>0.$ (In general, we say that a kernel $K$ taking values in $\Hom(V,V)$ is positive definite if, 
for any $z_1,\ldots , z_n$ in $\mathcal D$ and $v_,\ldots , v_n$ in $V,$
$$
\sum_{j,k=1}^n \langle {K(z_j,w_k)v_k},{v_j} \rangle\geq 0
$$
holds.) 
\begin{rem}\label{3.3}
When $\mathcal D$ is the unit ball  in $\mathbb C^2,$ the conditions of Theorem \ref{3.1} are satisfied for any indecomposable $\mathbb E^y.$ Furthermore, the spaces $\mathcal H^0$ and $\mathcal H^y$ are equal as sets. This follows from Theorem \ref{3.2} and the closed graph theorem.  
\end{rem}
\section{Homogeneous Cowen-Douglas pairs} For any bounded domain $\mathcal D \subseteq \mathbb C^n,$ the $n$-tuple $\boldsymbol T=(T_1,\ldots , T_n) $ of bounded linear operators on a Hilbert space $\mathcal H$ is said to be homogeneous (relative to the holomorphic automorphism group $\Aut(\mathcal D)$) if the joint (Taylor) spectrum of $\boldsymbol T$ is in $\overline{\mathcal D}$ and for every $g$ in $\Aut(\mathcal D),$
the $n$-tuple $g(\boldsymbol T)= g(T_1, \ldots ,T_n)$ is unitarily equivalent to $\boldsymbol T.$

Another important class of $n$-tuples of commuting operators associated with the domain $\mathcal D \subseteq \mathbb C^n$ is the extended Cowen-Douglas class 
${\mathrm B}_k^\prime(\mathcal D).$ Its elements are $n$-tuples $\boldsymbol T$ consisting of commuting operators $T_1, \ldots , T_n$ acting on some Hilbert space $\mathcal H$ such that (i) for every $z_1, \ldots , z_n$ in $\mathcal D,$ the joint kernel of $T_j-z_j$ ($1\leq j \leq n$) has the same dimension $k (\not = 0)$ and (ii) these joint kernels together generate $\mathcal H.$ The strict Cowen-Douglas class ${\mathrm B}_k(\mathcal D),$ as 
originally defined, satisfies a further ``closed range'' condition.  

The class ${\mathrm B}_k^\prime(\mathcal D)$ has practically all the good properties of ${\mathrm B}_k(\mathcal D).$ In particular for any $\boldsymbol T$ in ${\mathrm B}^\prime_k(\mathcal D)$ acting on the Hilbert space $\mathcal H,$ there exists a holomorphic Hermitian vector bundle and a Hilbert space $\hat{\mathcal H}$ of its sections with a reproducing kernel such that the commuting $n$-tuple $(M_1^*, \ldots, M_n^*),$ which is the adjoint of the $n$-tuple of multiplication by the coordinate functions on $\hat{\mathcal H},$ is unitarily equivalent to $\boldsymbol T$ under a unitary isomorphism of the Hilbert spaces $\mathcal H$ and $\hat{\mathcal H}.$ 

We wish to investigate, for bounded symmetric domains $\mathcal D,$ the homogeneous $n$-tuples in ${\mathrm B}^\prime_k(\mathcal D).$ For the case of the unit disc, there is a complete description and classification of these in \cite{KM}. (In that case, it turns out that the homogeneous operators  
in ${\mathrm B}^\prime_k(\mathbb D)$ are the same as in  ${\mathrm B}_k(\mathbb D).$)

\begin{thm}
For any bounded symmetric $\mathcal D,$ an $n$-tuple $\boldsymbol T$ in ${\mathrm B}^\prime_k(\mathcal D)$ is homogeneous if and only if the corresponding holomorphic Hermitian vector bundle is homogeneous under $\tilde{G}.$ 
\end{thm}
The proof (not entirely trivial) is the same as in \cite[Theorem 2.1]{KM}. 

For a bounded symmetric $\mathcal D,$ we call a $n$-tuple $\boldsymbol T$ in ${\mathrm B}^\prime_k(\mathcal D)$ and its corresponding bundle $E$ basic if $E$ is induced by an irreducible $\rho.$ From the results of Section 3, when $\mathcal D$ is the unit ball in $\mathbb C^n,$ $E$ is basic if and only if it is induced by some $\chi_\lambda\otimes \sigma$ with $\lambda < \sigma_\lambda.$

\begin{thm}
If $\mathcal D$ is the unit ball in $\mathbb C^2,$ all homogenous pairs in ${\mathrm B}^\prime_k(\mathcal D)$ are similar to direct sums of basic homogeneous pairs.
\end{thm}
The proof is based on Remark \ref{3.3}. The similarity arises as the identity map between $\mathcal H^0$ to $\mathcal H^y,$ which clearly intertwines the operators $M_j$ on the respective Hilbert spaces.  


\begin{thebibliography}{99}
\bibitem{HW} R. A. Herb and J. A. Wolf, \emph{Wave packets for the relative discrete series. I. The holomorphic case}, J. Funct. Anal. 73 (1987), 1 - 37.
\bibitem{Jak} H. P. Jakobsen, \emph{The last possible place of unitarity for certain highest weight modules}, Math. Ann. 256 (1981), 439 - 447.
\bibitem{KM} A. Kor\'{a}nyi and G. Misra, \emph{A classification of homogeneous operators in the Cowen-Douglas class}, Adv. Math., 226 (2011) 5338 - 5360.
\bibitem{S}
I. Satake, {\emph Algebraic structures of symmetric domains,}  Princeton University Press, Princeton, N.J., 1980.
\end{thebibliography}
\end{document}